\documentclass{amsart}

\usepackage{hyperref}

\usepackage{esint}
\usepackage{amssymb}
\usepackage{tikz}
\usepackage{graphicx}
\usepackage{amsrefs,enumitem}

\usepackage{fancyhdr}
\usepackage{time}
 
\newtheorem{theorem}{Theorem}
\newtheorem{lemma}[theorem]{Lemma}
\newtheorem{corollary}[theorem]{Corollary}
\newtheorem{proposition}[theorem]{Proposition}
\theoremstyle{definition}

\newcommand{\eref}[1]{(\ref{e.#1})}
\newcommand{\tref}[1]{Theorem \ref{t.#1}}
\newcommand{\lref}[1]{Lemma \ref{l.#1}}

\newcommand{\cref}[1]{Corollary \ref{c.#1}}

\newcommand{\sref}[1]{Section \ref{s.#1}}

\newcommand{\Z}{\mathbb{Z}}
\newcommand{\R}{\mathbb{R}}

\newcommand{\E}{\mathbb{E}}

\newcommand{\grad}{\nabla}

\def\XXint#1#2#3{{\setbox0=\hbox{$#1{#2#3}{\int}$ }
\vcenter{\hbox{$#2#3$ }}\kern-.6\wd0}}

\newcommand{\ep}{\varepsilon}

\renewcommand{\P}{\mathbb{P}}

\begin{document}

\title{Recovering coercivity for the G-equation in general random media}

\author{William M Feldman}
\address{Institute for Advanced Study, 1 Einstein Dr, Princeton, NJ 08540}
\email{wfeldman@math.ias.edu}
\maketitle

\begin{abstract}
The G-equation is a popular model for premixed turbulent combustion.  Mathematically it has attracted a lot of interest in part because it is a simple example of a Hamilton-Jacobi equation which is only coercive `on average'.  This paper shows that, after an almost surely finite waiting time, coercivity is recovered for the G-equation in a small mean, incompressible, space-time stationary ergodic velocity field.  The argument follows ideas from recent work of Burago, Ivanov and Novikov \cite{BuragoIvanovNovikov0,BuragoIvanovNovikov}, while significantly weakening the assumption on the velocity field.  The waiting time is explicitly characterized in terms of the space-time means of the velocity field and so mixing estimates on the waiting time can easily be derived. Examples are provided.
\end{abstract}

\section{Introduction}

The G-equation is
\begin{equation}\label{e.geqn}
 u_t = A|\grad u| + V(t,x) \cdot \grad u \ \hbox{ in } \ \R^d \times [t_0,\infty) \  \hbox{ with } \ u(x,t_0) = u_0(x)
 \end{equation}
where $V$ is a space-time stationary ergodic random field on $\R_t \times \R^d_x $ that is Lipschitz continuous with $|V| \leq M$,
\[ \grad \cdot V(t,x) = 0, \ \hbox{ and } \ |\E[V](t,x)| <A.\]
This is a simple model for premixed turbulent combustion. In this interpretation, the super-level sets of $u$ are ``burnt regions" and the sub-levels are ``unburnt regions", $V$ models a turbulent fluid flow, and $A$ is the laminar flame speed.  Usually $u$ is called $G$ in the applied literature, which explains the name of the equation.

In mathematical terms this is a (geometric) Hamilton-Jacobi equation with convex Hamiltonian $H(p,t,x) = A|p| + V(t,x) \cdot p$.  The difficulty of the problem comes from the lack of coercivity: it may be that $M \gg A$.  The key consequences of coercivity are Lipschitz estimates (in the time independent case) and reachability estimates for controlled trajectories (in general).  These estimates, derived from coercivity, play a fundamental role in homogenization results for Hamilton-Jacobi equations, but they are not present for the G-equation.  Nonetheless, the formal intuition is that the Hamiltonian associated with the G-equation is ``coercive on average" since $\E H(p,t,x) = |p| + \E[V] \cdot p$ is coercive.  Of course, one cannot just take expectations on both sides of \eref{geqn} and hope to derive something since $V$ and $\grad u$ are not independent.  Nonetheless, as we will show here, the primary consequences of coercivity (Lipschitz/reachability estimates) are indeed recovered at the length/time scale $T(t,x)$ (a stationary random field) where the space-time averages of $V$ centered at $(t,x)$ become less than $A$.

We put the above in more precise terms.  The G-equation \eref{geqn} has a natural control interpretation with trajectories
\[ \dot{X}_t = V(t,X_t) + \alpha_t \ \hbox{ with any measurable control } \ |\alpha_t| \leq A.\]
It turns out that
\begin{equation}\label{e.controlformula}
 u(t,x) = \sup_{x\in R_t(t_0,x_0)} u_0(t_0,x_0),
 \end{equation}
where $R_t$ is called the reachable set, defined for $t \in \R$,
\[ R_{t}(t_0,x_0) = \left\{ x \in \R^d: 
\begin{array}{c}
\hbox{ there exists a controlled trajectory $X$ on $[t_0,t]$}\\
\hbox{with $X_{t_0} = x_0$ and $X_t = x$}
\end{array}\right\}. \]
 The reachable set from a given space-time point is the main object of interest in this study.  The indicator function ${\bf 1}_{R_t(t_0,x_0)}(t,x)$ is a special solution of \eref{geqn}, in PDE language it is like a nonlinear version of a fundamental solution.

We say that there is a finite waiting time if there is a stationary random field $T: \R^d \times \R \to [0,\infty)$  that is finite almost surely and for which the following delayed coercivity condition holds: there exists $c>0$ universal such that,
\begin{equation}\label{e.waitcoercivity}
 B_{c(A-|\E[V]|)|t-t_0|}(x) \subset R_{t}(t_0,x_0) \ \hbox{ for all } \ |t -t_0|\geq T(t_0,x_0). 
 \end{equation}
In the time independent case $V(t,x) = V(x)$, by some simple manipulations of the control formula \eref{controlformula} using \eref{waitcoercivity}, it follows that solutions of the G-equation are Lipschitz continuous at length/time scales larger than the waiting time
\[ |u(t,x) - u(y,s)| \leq \|\grad u_0\|_\infty[2MT(x)\vee T(y) +  |x-y|+M|t-s|].\]
Thus, this can be thought of as a large scale regularity property. In general such results play an important role in quantitative homogenization theory.  The waiting time estimate \eref{waitcoercivity} has been proved previously in space-time periodic \cite{XinYu,CardaliaguetNolenSouganidis},  stationary ergodic (time independent) \cite{NolenNovikov,CardaliaguetSouganidis}.  Recently, Burago, Ivanov and Novikov \cite{BuragoIvanovNovikov} proved \eref{waitcoercivity} in space-time uniformly ergodic environments, a class which at least includes periodic, almost periodic, and some finite range dependence random velocity fields with special structure.  We give the first proof of \eref{waitcoercivity} in the most general setting for homogenization theory, space-time stationary ergodic random environments, building on the new ideas of \cite{BuragoIvanovNovikov}.  This also gives a new proof of finite waiting time in time independent media, which was proved some time ago by Cardaliaguet and Souganidis \cite{CardaliaguetSouganidis}.

In what follows we make some simplifications and consider only the case $A = 1$ and $\E[V] = 0$.  The general case presented above can be recovered by a rescaling of the time variable and ``trading" some of the control to make $\E[V] = 0$.

The starting point to understanding the spreading of the reachability set $R_\tau$ follows from the divergence free condition and the isoperimetric inequality.  Integrating \eref{geqn} over $\R^d$, since ${\bf 1}_{R_t}$ itself is a solution of the G-equation,
\[ \frac{d}{dt}|R_t| = \int_{\partial R_t} 1 + V(t,x) \cdot n \ dS =  |\partial R_t| \geq d\omega_d^{1/d}|R_t|^{1-\frac{1}{d}}.\]
Integrating this differential inequality from $t_0$ to $t$ yields
\begin{equation}\label{e.Rdvolgrowth}
 |R_t| \geq \omega_d (t-t_0)^d.
 \end{equation}
In combination with the uniform upper bound $M$ on the speed of trajectories one can obtain
\[ |R_t(t_0,x_0) \cap B_{M(t-t_0)}(x_0)| \geq \omega_d(t-t_0)^d.\]
This estimate, however, contains no information on how the reachable set spreads.

In the below we write
\[ \Box_r(x) = x+(-\tfrac{r}{2},\tfrac{r}{2})^d \]
Let us consider a localization of the reachability set growth estimate on a large open box $\Box_r = \Box_r(x_0)$
\[ \frac{d}{dt} |\Box_r \cap R_t |_d =|\Box_r \cap \partial R_t|_{d-1} - \textup{flux}(V,R_t \cap \partial \Box_r).  \]
If the flux term was not present the relative isoperimetric inequality would allow to show that $\Box_r$ is completely filled by $R_t$ in time proportional to $r$.  Thus the issue lies with the flux through $R_t \cap \partial \Box_r$.

The clever new arguments introduced by Burago, Ivanov, and Novikov \cite{BuragoIvanovNovikov} show how to control this flux in terms of only the uniform convergence of the spatial averages.  Given $\ep>0$ define
\begin{equation}\label{e.uniformmean}
 r^*_\ep = \sup \left\{ r >0: \ \sup_{x,t} |\frac{1}{|\Box_r|_d}\int_{\Box_r(x)} V(t,x) \ dx| \geq \ep\right\}.
 \end{equation}
There is a universal $\ep_0(d,M)$ such that if $r^*_{\ep_0}<+\infty$ then \cite{BuragoIvanovNovikov} obtain a finite waiting time $T$ (independent of $t,x$) such that \eref{waitcoercivity} holds.  This is similar to, although slightly weaker than, uniform ergodicity, the condition that the ergodic averages converge to the mean uniformly in space-time.  This makes heuristic sense since we imagine that the problem is coercive on average at length scale $r^*_{\ep_0}$.  This condition does hold for periodic, almost periodic, and also on a non-trivial class of finite range of dependence random velocity fields $V$ \cite[Remark 6.5]{BuragoIvanovNovikov}.  Still the condition is fairly restrictive in the context of random media, such uniform space-time convergence of the ergodic averages contained in the condition $r^*_\ep<+\infty$ for small $\ep>0$ will definitely not hold in general even on random velocity fields with good mixing properties.  

In the present paper we greatly generalize the class of velocity fields to which these methods are applicable. We are able to prove \eref{waitcoercivity} in the class of space-time stationary ergodic velocity fields, this is the absolute weakest assumption which is widely used in studying homogenization.  We will only need to control a much weaker quantity than \eref{uniformmean}. First define the space-time boxes
\[ Q_r(t,x) = (t,x) + (-\tfrac{r}{2},  \tfrac{r}{2}) \times \Box_r . \]
Fix $N> 1$ and define the empirical averages
\begin{equation}\label{e.en}
 E_N[V;Q_r] = \sup_{(k,n), \in \Z \times \Z^d , \ |(k,n)|_\infty < N/2}\left|\frac{1}{|Q_{r/N}|}\int_{Q_{r/N}(k\frac{r}{N},n\frac{r}{N})} V(t,x) \ dxdt\right|.
 \end{equation}
Define $E_N[V,Q_r(t,x)]$ analogously for $(t,x) \in \R_t \times \R^d_x$.  From the ergodic theorem and $\E[V] = 0$ the following limit holds on an event of full probability
\[\lim_{r \to \infty} E_N[V;Q_r(t,x)] = 0 \ \hbox{ for all } \ (t,x) \in \R_t \times \R^d_x.\]
Then define
\begin{equation}\label{e.rnequantity}
 r^*_{N,\ep}(t,x) = \sup \left\{ r >0: \  E_N[V;Q_r(t,x)] \geq \ep\right\}.
 \end{equation}
This quantity is stationary in $(t,x)$ and, by the ergodic theorem (see Akcoglu and Krengel \cite{AkcogluKrengel}), it is finite almost surely for every $\ep >0$.

We show that $r^*_{N,\ep}$ controls the waiting time for a sufficiently large, universal, $N$.
\begin{theorem}\label{t.main}
Suppose that $V$ is a space-time stationary ergodic random field, uniformly bounded $\|V\|_\infty \leq M$, and Lipschitz continuous.  Then there are dimensional constants $c(d),C(d)>0$ such that
\[ B_{c|t-t_0|}(x_0) \subset R_{t}(t_0,x_0) \ \hbox{ for all } \ |t -t_0|\geq T(t_0,x_0):= Cr^*_{N,\ep}(t_0,x_0) \]
for $\ep = C^{-1}M^{-(d-1)}$ and $N = \lceil C M^{5d+2}\rceil$.  In particular $T$ is finite almost surely.
\end{theorem}
Since the dependence of $\ep$ and $N$ on $M$ is quite explicit, this would probably allow to consider unbounded velocity fields with finite moments, at least under some mixing conditions.

This bound also allows to naturally derive mixing estimates and tail bounds on the waiting time if we assume mixing conditions on $V$.   In particular one can get an explicit upper bound (depending on $M$) for the length scale where coercivity first holds with high probability.  We give an example result in this direction.  The statement will be slightly imprecise, since we do not want to explicitly define this mixing condition yet.  When we say a constant is universal below, we mean it depends at most on the dimension and the hidden constants in the mixing condition, see \sref{mixing} for a precise description.
\begin{corollary}\label{c.mixing}
Suppose that $V$ satisfies an $\alpha$-mixing condition with stretched exponential decay with stretching exponent $\beta>0$ and unit length scale.  Then there are universal constants $c,C>0$ such that
\[\P(T(t,x) \geq t) \leq C\exp\left[-c(\tfrac{t}{\ell(M)})^{\beta'}\right]\]
and $T$ is $\alpha$-mixing with stretched exponential decay with exponent $\beta' = \frac{(d+1)\beta}{d+1+\beta}$ and length scale $\ell(M) = CM^{2(d-1)/\beta'}(1+|\log M|)^C$.
\end{corollary}


We expect that this quantitative regularity will have applications for quantitative homogenization of the G-equation especially in the time independent case.

\subsection{Literature} 
The G-equation was introduced by Williams \cite{Williams}, it is a popular simple model for flame propagation in turbulent combustion \cite{Peters}.  In the mathematical community popular questions have been related to homogenization \cite{XinYu,CardaliaguetNolenSouganidis,NolenNovikov,CardaliaguetSouganidis,BuragoIvanovNovikov0,BuragoIvanovNovikov,henderson2018brownian} and quantifying front speed enhancement as $M \to \infty$ \cite{MR2847532,MR3132416,MR3228465,MR3237880,MR3549020,MR3580814}.  The topic of asymptotic flame speed enhancement has lead to some very interesting mathematics, including connections with dynamical systems, however this direction seems to be less relevant to the present paper, so we will focus on explaining the works studying homogenization.

In the homogenization results listed below more and more general coercivity estimates have been developed progressively.  What we want to emphasize about our coercivity estimate, in comparison to previous results, is that it allows the most general assumption on the random media while still having explicit dependence on the random field $V$.

The first works on homogenization of the G-equation were by Cardaliaguet, Nolen and Souganidis \cite{CardaliaguetNolenSouganidis} (considering space-time periodic media) and, at the same time, by Xin and Yu \cite{XinYu} (considering time independent periodic).  In stationary ergodic media (time independent), $d=2$, Nolen and Novikov \cite{NolenNovikov} proved homogenization assuming the existence of a stream function with a certain growth condition, this would follow from a sufficiently strong mixing condition on the field.  A key step there is a waiting time estimate, their proof strongly uses the $2$-$d$ structure (scalar stream function, periodic trajectories are boundaries of open sets).   Their bound on the waiting time does explicitly depend on the spatial averages of $V$ via the stream function.  Next Cardaliaguet and Souganidis \cite{CardaliaguetSouganidis} obtained a very general homogenization result, covering stationary ergodic media in all dimensions in the time independent case.  As an important step they proved a new waiting time bound in stationary ergodic (time independent) media, the proof is abstract, using ergodicity, so, unlike our result, the dependence of the waiting time on $V$ is not explicit.  Finally we come to the works of Burago, Ivanov and Novikov \cite{BuragoIvanovNovikov0,BuragoIvanovNovikov}, which, as explained above, are the inspiration for the present work.  Their main results were the delayed coercivity condition \eref{waitcoercivity} under the uniform convergence of the means \eref{uniformmean}, and a homogenization result in space-time finite range dependence media with a special structure.  Our new result, building on \cite{BuragoIvanovNovikov}, is the first bound on the waiting time in the most general setting of space-time stationary ergodic media.

\subsection{Acknowledgments}  Thank you to Takis Souganidis for many helpful conversations and for bringing the result \cite{BuragoIvanovNovikov} to my attention. Thank you to Inwon Kim and Chris Henderson for helpful comments on the manuscript.

\subsection{Support} The author appreciates the support of the Friends of the Institute for Advanced Study and NSF RTG grant DMS-1246999.  

\section{Notation}
  For \sref{bdryflux} the only relevant parameters of the problem, which will appear in the estimates, are $d$ and $M = \|V\|_\infty$.  Dependence on $d$ will be omitted in general, we write $c$, $C$ for positive constants depending at most on the dimension which may change from instance to instance.  All dependence on $M$ will be made explicit.  In \sref{mixing} we will introduce some additional parameters related to the mixing condition, the dependence of constants $c,C$ on these parameters will also be omitted.
  
  We will need to measure various co-dimension sets $E$ of $\R^n$.  We will denote $\mathcal{H}^m$ for the $m$-dimensional Hausdorff measure.  We usually write $|E|_m = \mathcal{H}^m(E)$ when it is not too confusing.  We will also make use of the perimeter, for an open $\Omega \subset \R^n$ and a set $E \subset \R^n$
  \[ \textup{Per}(E,\Omega) = \sup \{ \int_{E \cap \Omega} \grad \cdot \varphi \ dx: \ \varphi \in C^1_c(\Omega) \ \hbox{ and } \ |\varphi| \leq 1\}.\]
  This can be defined also on closed sets $K$ as the infimum of $\textup{Per}(E,\Omega)$ over open $\Omega \supset K$. It can also be defined similarly for $\Omega$ in a flat $m$-dimensional slice of $\R^n$ or a finite union of such.  We will use this to compute perimeters on $d$-dimensional boundaries of boxes in dimension $d+1$.  See \cite{BuragoIvanovNovikov} for more details and references on the geometric measure theoretic tools needed for the G-equation.
\section{General strategy}
This section outlines the general strategy of \cite{BuragoIvanovNovikov} to integrate the local volume growth ODE
\begin{equation}\label{e.volumegrowth}
 \frac{d}{dt} |\Box_r(x_0) \cap R_t |_d =|\Box_r(x_0) \cap \partial R_t|_{d-1} - \textup{flux}(V,R_t \cap \partial \Box_r(x_0)).  
 \end{equation}
 The argument proceeds in three steps: Step 1 filling a small fraction $\alpha|\Box_r(x_0)|$, Step 2 filling from $\alpha|\Box_r(x_0)|$ to $(1-\alpha)|\Box_r(x_0)|$, Step 3 filling the small complement.

Step 1: Set $t_1 = t_0 + \frac{r}{2M}$, then we claim
\begin{equation}\label{e.alphadef}
 |\Box_r \cap R_t|_d \geq \alpha |\Box_r|_d \ \hbox{ with } \ \alpha = \frac{\omega_d}{(2M)^d}.
 \end{equation}
By the control formula and $|V| \leq M$ we have $R_{t} \subset \Box_r$ for $t_0 \leq t \leq t_1$.  Thus $|R_t| = |\Box_r \cap R_t|_d$ and \eref{Rdvolgrowth} proves the claim.

 Step 2: If the boundary flux term could be ignored, then we could just integrate \eref{alphadef} using the relative isoperimetric inequality:
 \begin{equation}\label{e.relisoperimetric}
 |\Box_r \cap \partial R_t|_{d-1} \geq \lambda_1(d) \min\{ |\Box_r \cap R_t|_d, | \Box_r \setminus R_t|_d\}^{\frac{d-1}{d}}.
 \end{equation} 
 The central difficulty is to show that the boundary flux is appropriately small, and this is where the beautiful new ideas of \cite{BuragoIvanovNovikov} come in.  In \sref{bdryflux} we show how to modify those ideas to handle a much more general class of velocity fields.
 
 Step 3:  Suppose that $|R_{t_2} \cap \Box_{2r}|_d \geq (1-2^{-d}\alpha)|\Box_{2r}|_d$ at some time $t_2$.  Let $y_0 \in \Box_r$.  Then by step 1 above
 \[ |R_{-\frac{r}{2M}}(t_2+\tfrac{r}{2M},y_0) \cap \Box_{2r}|_d \geq \alpha |\Box_r|_d \geq 2^{-d}\alpha |\Box_{2r}|_d.
 \]
 Thus, for every $y_0 \in \Box_{r}(x_0)$
 \[ R_{-\frac{r}{2M}}(t_2+\tfrac{r}{2M},y_0) \cap R_{t_2}(t_0,x_0) \neq \emptyset\]
 and so 
 \[ \Box_r(x_0) \subset R_{t_2+\frac{r}{2M}}(t_0,x_0).\]

 Thus the proof of \tref{main} will be complete if we can prove the following proposition.
 \begin{proposition}
 Suppose that $r \geq r^*_{N,\ep}(t,x)$ with $\ep = c(d)M^{-(d-1)}$ and $N = C(d)M^{5d+2}$.  Call $t_1(\Box_r)$ and $t_2(\Box_r)$ as defined above ($t_2 = +\infty$ if such time does not exist).  Then $t_2(\Box_r)$ exists and
 \[ t_2 - t_1 \leq C(d)r^d.\]
 \end{proposition}

\subsection{Regularity} The regularity of the reachable set is already discussed in \cite{BuragoIvanovNovikov} and \cite{CardaliaguetSouganidis}, from \eref{volumegrowth} one can derive that $R_t$ is a finite perimeter set for almost every time.  For our purposes we need a bit more. In particular we need to understand the regularity of the space-time reachable set $R$, of which $R_\tau = R \cap \{t = \tau\}$ are the time slices.

We argue formally here, and justify below by regularization, also we consider only the case $t \geq t_0$, the case $t \leq t_0$ follows by symmetry.  The indicator function of the reachable set ${\bf 1}_{R_t}$ solves the G-equation in the viscosity sense
\[ \partial_t {\bf 1}_{R_t} = |\grad {\bf 1}_{R_t}| + V(t,x) \cdot \grad {\bf 1}_{R_t}  \ \hbox{ in } \ \R^d \times [t_0,\infty)\]
with initial data $R_{t_0} = \{x_0\}$.  First, from the control interpretation, it is immediate that
\[ R_t \subset B_{(1+M)(t-t_0)}(x_0).\]
Taking absolute values on both sides of the G-equation
\begin{equation}\label{e.1+M}
 |\partial_t {\bf 1}_{R_t}|  \leq (1+M) |\grad {\bf 1}_{R_t}|.
 \end{equation}
If instead we integrate the G-equation over $\R^d$ and use the divergence theorem
\[ \partial_t|R_t|_d = \int_{\R^d} |\grad{\bf 1}_{R_t}|.\]
Integrating in time
\[ (1+M)^d(t-t_0)^d \geq |R_t|_d - |R_{t_0}|_d = \int_{t_0}^t\int_{\R^d} |\grad{\bf 1}_{R_\tau}| \ dxd\tau\]
Then using the pointwise inequality \eref{1+M} we obtain also
\[ \int_{t_0}^t\int_{\R^d} |\partial_t{\bf 1}_{R_\tau}|+ |\grad_x{\bf 1}_{R_\tau}| \ dxd\tau \leq (1+M)^{d}(2+M)(t-t_0)^d.\]
Summarized we have the following result:
\begin{lemma}
For each $(t_0,x_0)$ the reachable set $R(t_0,x_0)$ is a finite perimeter set of $\R_t \times \R^d_x$ with
\[ \textup{Per}(R(t_0,x_0),[t_0-\tau,t_0+\tau] \times \R^d)\leq 2(1+M)^{d}(2+M)\tau^d \ \hbox{ for all } \ \tau \geq 0.\]
\end{lemma}
\begin{proof}
Apply all the above arguments to the solution of the G-equation with initial data
\[ u^\delta(t_0,x) = \varphi(|x-x_0|/\delta)\]
where $\varphi$ smooth, $\varphi(0) =1$, and $\varphi$ is supported in $B_1(0)$.  Then, by the control formulation, $u^\delta$ converges pointwise to ${\bf 1}_{R_t}$.  The bound of the Lemma holds for the space-time BV norm of $u^\delta$, since the BV norm is lower semi-continuous for sequences converging in $L^1$ the result is obtained.
\end{proof}
We mention one other piece of regularity information, which is a lower continuity estimate on $|\Box_r \cap R_t|_d$.
\begin{lemma}\label{l.lowercont}
For each $(t_0,x_0)$, each $\Box_r$, and $t \geq t_0$
\[ \frac{d}{dt}|\Box_r \cap R_t(t_0,x_0)|_d \geq -C(d)Mr^{d-1}.\]
\end{lemma}
\begin{proof}
Follows from \eref{volumegrowth} bounding the term $|\Box_r \cap \partial R_t|_{d-1} \geq 0$ and the term $|\textup{flux}(V,R_t \cap \partial \Box_r)| \leq M |\partial \Box_r|_{d-1} = 2dMr^{d-1}$.  Again do a regularization as in the previous proof to make this rigorous.
\end{proof}

\section{Controlling the boundary flux}\label{s.bdryflux}
The following arguments are adaptations of \cite{BuragoIvanovNovikov}.  The aim is to control the flux term using that that space-time averages of $V$ are small at the length scale $r\geq r^*_{N,\ep}$ (at least down to scale $r/N$).  We provide a heuristic description and then move to formal statements.   The first claim is that the averages of $V$ on $d$-dimensional boundary faces of $\Box_r \times [t-r,t+r]$ are small.  This is not true at scale $r/N$, but at some intermediate scale $Lr/N$ it follows from the mean value theorem and the divergence free condition.  This is made precise in the following Lemma.

\begin{lemma}\label{l.faceflux}
Suppose that $r \geq r^*_{N,\ep}(t,x)$, let $1 \leq L < N$ an integer and $F$ be a $(d-1)$-dimensional cube contained in $\Box_r(x)$ with side lengths $LN^{-1}r$ and $I$ a time interval of length at least $LN^{-1}r$ contained in $[t-r/2,t+r/2]$, then
\[\left|\int_I\textup{flux}(V,F)\ d \tau \right | \leq C(d)(\ep+ML^{-1}) |F\times I|_{d+1}\]
\end{lemma}
The next issue is that we are not looking at $\textup{flux}(V,\partial \Box_r \times [t-r,t+r])$ but at $\textup{flux}(V,R \cap (\partial \Box_r \times [t-r,t+r]))$.  Looking in $LN^{-1}r$ size sub-faces tiling the boundary we see that if $R$ takes up most of the measure of the face, or only a small portion then there is no problem.  The only issue is on the sub-faces where $R$ and $R^C$ both take up a nontrivial portion of the measure.  However, in this case, by the relative isoperimetric inequality there must be a corresponding proportion of  the total perimeter $\textup{Per}(R,\partial \Box_r \times I)$.  This allows to control the total flux through $R \cap (\partial \Box_r \times [t-r,t+r])$ by the total flux (already small) plus a term involving the perimeter $\textup{Per}(R,\partial \Box_r \times I)$.  This is not precisely how the argument goes, there are some nice tricks which were introduced by \cite{BuragoIvanovNovikov}, which we re-use.  

We make a technical note before stating the Lemma. It is convenient to establish our bounds actually on a space-time rectangle $\Box_r \times [t_0-\gamma r, t_0 + \gamma r]$, with a dimensional constant $\gamma = 1 + \frac{2d}{\lambda_1(d)}$ where $\lambda_1(d)$ is the constant of the relative isoperimetric inequality in the cube as in \eref{relisoperimetric}.
\begin{lemma}\label{l.bdryflux}
Let $[t-\gamma r,t+\gamma r] \times \Box_r(x) $ be a space-time rectangle with side length $r \geq r^*_{N,\ep}(t,x)$, $L \in \{1,\dots,N\}$, and $I$ be a subinterval of $[t-\gamma r,t+\gamma r]$ of length $LN^{-1}r$.  Then
\[\left|\int_I\textup{flux}(V,R_\tau \cap \partial \Box_r(x)) \ d\tau\right| \leq C(d)M{L}{N}^{-1} r\textup{Per}(R,\partial \Box_r \times I) +C(d)(\ep+ML^{-1}) | I \times \partial \Box_r|_d.\]

\end{lemma}

It is not a-priori obvious but it turns out to be optimal to choose $L = \lceil M\ep^{-1} \rceil$.  First the proof of \lref{bdryflux} using \lref{faceflux}

\begin{proof}[Proof of \lref{bdryflux}]
Let $F$ be $(d-1)$-dimensional sub-face of $\partial \Box_r(x)$ with with side lengths $LN^{-1}r$ and $I$ a subinterval of $[t-\gamma r,t+\gamma r]$ of the same length.  First note that
\[ \textup{flux}(V,F) = \textup{flux}(V,F \cap R_\tau) + \textup{flux}(V,F \setminus R_\tau),\]
and so, applying \lref{faceflux} on with radius $\gamma r$, and by reverse triangle inequality
\begin{align*}
 \left||\int_{I}\textup{flux}(V,F \cap R_\tau) \ d\tau | -|\int_I\textup{flux}(V,F \setminus R_\tau)\ d\tau|\right| &\leq |\int_I\textup{flux}(V,F) \ d \tau | \\
 &\leq C(\ep+ML^{-1}) |F\times I|_d.
 \end{align*}
By the relative isoperimetric inequality on $F \times I$ (see \cite[Theorem A.5]{BuragoIvanovNovikov})
\begin{align*}
 \min\{|(F \times I) \cap R |_d,|(F \times I) \setminus R|_d\} &\leq CLN^{-1}r\textup{Per}(R,F \times I). 
 \end{align*}
Therefore
\[ |\int_I\textup{flux}(V,F \cap R_\tau)\ d \tau| \leq C(\ep+ML^{-1})  |F\times I|_d+CMLN^{-1}r\textup{Per}(R,F \times I).\]
Summing over a partition of $\partial \Box_r(x)\times I$ by sub-faces $F\times I$
\[ |\int_{I}\textup{flux}(V,\partial \Box_r(x) \cap R_\tau )\ d\tau| \leq C(\ep+ML^{-1}) |\partial \Box_r(x) \times I|_d+CMLN^{-1}r\textup{Per}(R,\partial \Box_r \times I)\]

\end{proof}

\begin{proof}[Proof of \lref{faceflux}]
Let $F$ as in the statement of the Lemma, we can assume that $F = [0,\frac{Lr}{N}]^{d-1} \times \{x_d = 0\}$ and $I = [0,\frac{rL}{N}]$.  Then $I \times F$ is contained in a union of $(L+1)^{d}$ of the $N^{d+1}$ space-time cubes of width $r/N$ partitioning $\Box_r \times [t-\frac{r}{2},t+\frac{r}{2}]$.  Call $P$ to be the union of the spatial projection of these rectangles,
\[ P = y+[-\tfrac{r}{2N},(L+\tfrac{1}{2})\tfrac{r}{N}]^{d-1}\times[-\tfrac{r}{2N},\tfrac{r}{2N}] \ \hbox{ for some } \ |y|_\infty \leq \tfrac{r}{2N}\]
and call $J$ to be the union of the temporal projections
\[ J = s + [-\tfrac{r}{2N},(L+\tfrac{1}{2})\tfrac{r}{N}] \ \hbox{ for some } \ |s| \leq \tfrac{r}{2N}.\]
By the definition of $r^*_{N,\ep}(t,x)$
\[ \left|\frac{1}{|J \times P|_d} \int_{J \times P} V(t,x) \cdot e_d \ dtdx\right| \leq \ep. \]
Then, by Fubini and the mean value theorem, there is a face $F' = P \cap \{x_d = h\}$ with $h \in y_d+[-\frac{r}{2N},\frac{r}{2N}]$ such that
\[ \left|\frac{1}{|J \times F'|_d} \int_{J \times F'} V(t,x) \cdot e_d \ dtdx\right| \leq \ep.\]
Applying the divergence theorem in the region $P \cap \{x_d \in [0,h]\}$ (assume $h>0$, the other case is symmetric) at each fixed time and using $\grad \cdot V = 0$
\begin{align*}
\left|\int_{J \times (P \cap \{x_d = 0\}) } V(t,x) \cdot e_d \ d\mathcal{H}^{d-1}\right| &= \left| \int_{J \times F'}V(t,x) \cdot e_d \ dtd\mathcal{H}^{d-1} + \int_{J \times (\partial P \cap \{0 < x_d < h\})} V(t,x) \cdot n \ dtd\mathcal{H}^{d-1}\right| \\
&\leq \ep |J \times F'|_d + CMh((L+1)r/N)^{d-2}|J| \\
&\leq (\ep+CM/L)|J \times F'|_d
\end{align*}
Finally $ |F'|- |F| \leq C\frac{1}{L}|F|$ and $ |J| - |I| \leq C\frac{1}{L}|I|$ and, similarly,
\[ \left|\int_{ J \times (P \cap \{x_d = 0\})} V(t,x) \cdot e_d \ dtd\mathcal{H}^{d-1}(x) - \int_{I \times F}V(t,x) \cdot e_d \ dtd\mathcal{H}^{d-1}(x)\right|\leq C\frac{M}{L}|I \times F|_d \]
and, combining, we get the result.

\end{proof}

\subsection{Volume growth differences}
Integrating the ODE \eref{volumegrowth} and using the flux bound \lref{bdryflux}, we have proven that for $r \geq r^*_{N,\ep}(t_0,x_0,\ep)$ and $I = [t',t]$ a subinterval of $[t_0,t_0+\gamma r]$ of length at least $Lr/N$ (with $L = \lceil M\ep^{-1}\rceil$, as it will be fixed for the rest of the section)
\begin{equation}\label{e.odeperRHS}
  |\Box_r(x_0) \cap R_t |_d - |\Box_r(x_0) \cap R_{t'}|_d \geq \int_{t'}^{t}|\Box_r \cap \partial R_\tau|_{d-1} \ d\tau -C\ep |I \times \partial \Box_r|_d -CM^2\ep^{-1}N^{-1}r\textup{Per}(R,I \times\partial \Box_r ).  
  \end{equation}
Our goal is the following estimate
\begin{lemma}\label{l.volumegrowthode}
Let $1>\ep>0$ and $M^{-1}\wedge\frac{1}{4} > \beta > 0$ fixed, let $N= \lceil \beta^{-2} \ep^{-3}M^3\rceil$, let $r \geq r^*_{N,\ep}(t_0,x_0)$, and let $I$ a subinterval of $[t_0,t_0+\gamma r]$ of length $|I| = \beta r$.  Then 
\begin{equation}\label{e.volode}
 |\Box_r(x_0) \cap R_t|_d - |\Box_r(x_0) \cap R_{t'}|_d \geq \int_{I}|\Box_r \cap \partial R_\tau|_d \ d \tau -C\ep r^{d-1}|I|_1 .
 \end{equation}
\end{lemma}
This is basically achieved by averaging over a small range of $r$ and applying the mean value theorem to find a value of $r$ where the term $\textup{Per}(R,I \times\partial \Box_r )$ is of ``typical" size.  In the proof we will apply the following co-area formula several times:
\begin{lemma}[Federer co-area formula]\label{l.coarea}
Let $f : \R^n \to \R$ be a Lipschitz function, and $E \subset \R^n$ be an $\mathcal{H}^k$-rectifiable set.  Then the function $\lambda \mapsto \mathcal{H}^{k-1}(E \cap f^{-1}(\lambda))$ is Lebesgue measurable, $E \cap f^{-1}(\lambda)$ is $\mathcal{H}^{k-1}$-rectifiable for Lebesgue a.e. $\lambda \in \R$ and
\[ \int_{E} |\grad_{\tan} f(y)| d \mathcal{H}^{k}(y) = \int_{\R} \mathcal{H}^{k-1}(E \cap f^{-1}(\lambda)) \ d\lambda\]
where $\grad_{\tan} f(y)$ is the component of $\grad f(y)$ tangent to $E$ at $y$.
\end{lemma}
\begin{proof}[Proof of \lref{volumegrowthode}]   By the co-area formula \lref{coarea} applied to $\partial R$ with $f(t,x) = |x|_\infty$ and using $|\grad_{\tan} f| \leq 1$, for some $\delta>0$ to be chosen,
\begin{equation}\label{e.dtod-1}
 \mathcal{H}^{d}((I \times \Box_{(1+\delta )r} ) \cap \partial R) \geq \int_{r}^{(1+\delta )r}\textup{Per}(R,I \times \partial \Box_\rho) \ d\rho.
 \end{equation}
Again from the co-area formula \lref{coarea} applied to $\partial R$ now with $f(t,x) = t$
\begin{equation}\label{e.coareat}
 \frac{1}{(1+M^2)^{1/2}}\mathcal{H}^{d}((I \times \Box_{(1+\delta)r}) \cap \partial R ) \leq \int_{I} \mathcal{H}^{d-1}(\Box_{(1+\delta)r} \cap \partial R_\tau) d\tau.
 \end{equation}
In more detail, since the normal direction $n = (n_t,n_x)$ to $\partial R$ at $(t,x)$ has
\[ |n_x| = \frac{1}{(1+|V(t,x) \cdot \hat{n}_x|^2)^{1/2}} \geq \frac{1}{(1+M^2)^{1/2}}\]
and so, using $\grad_{t,x} f(t,x) = (1,0)$,
\[ |\grad_{\tan} f|(t,x) =(1-|\grad f(t,x)\cdot n|^2)^{1/2} = (1-|n_t|^2)^{1/2} = |n_x| \geq \frac{1}{(1+M^2)^{1/2}}. \]
Plugging this inequality into the co-area formula \lref{coarea} gives \eref{coareat}.

Next estimate $\int_{I} \mathcal{H}^{d-1}(\Box_{(1+\delta)r} \cap \partial R_\tau) \ d\tau$, as in \cite[Lemma 4.2]{BuragoIvanovNovikov}, by integrating \eref{volumegrowth} on $I$ and bounding the flux term by $M|\partial \Box_r|_{d-1}$ and the volume difference by the total volume of $\Box_r$:
\begin{equation}\label{e.totalperimeter}
r^d+CMr^{d-1}|I| \geq \int_{I} \mathcal{H}^{d-1}(\Box_{(1+\delta)r} \cap \partial R_\tau) \ d\tau.
\end{equation}
Since $|I| \leq M^{-1}r$ the left hand side is bounded above by $C r^d$.

Combining the previous inequalities \eref{totalperimeter}, \eref{coareat}, and \eref{dtod-1}
\[ CMr^d \geq \int_{r}^{(1+\delta)r} \textup{Per}(R,I \times \partial \Box_\rho) \ d\rho\]
and so, for some $r \leq \rho \leq (1+\delta)r$
\[ \textup{Per}(R,I \times \partial \Box_\rho) \leq CM\delta^{-1}r^{d-1}.\]
Plugging this into the difference equation \eref{odeperRHS}
\begin{align*}
  |\Box_\rho \cap R_t |_d- |\Box_\rho \cap R_{t'} |_d&\geq\int_{I}|\Box_\rho \cap \partial R_\tau|_{d-1}\ d\tau -C\ep |\partial \Box_\rho|_{d-1} |I|_1 -CM^2\ep^{-1}N^{-1}r\textup{Per}(R,I \times \partial \Box_\rho ).  \\
  &\geq \int_{I}|\Box_r \cap \partial R_\tau|_{d-1}  \ d\tau -C\ep r^{d-1}|I|_1-CM^3\delta^{-1}\ep^{-1}N^{-1}r^{d}.
  \end{align*}
  Then using
  \[ ||\Box_\rho \cap R_t |_d - |\Box_r \cap R_{t}|_d| \leq |\Box_\rho \setminus \Box_r|_d \leq C\delta r^d\]
  and choosing $\delta = \beta \ep$ and $N = \beta^{-2}M^3\ep^{-3}$ to match the size of all the error terms
  \[ |\Box_r \cap R_t |_d- |\Box_r \cap R_{t'} |_d\geq\int_{I}|\Box_r \cap \partial R_\tau |_{d-1} \ d\tau - C\ep r^{d-1}|I|_1.\]
  Note that with these choices $LN^{-1} =\beta^2\ep^2M^{-2} \leq \beta$ using $\beta \leq 1/4$, $M \geq 1/2$ and $\ep < 1$, thus the application of \eref{odeperRHS} above was justified since $|I|_1 = \beta r \geq LN^{-1}r$.
\end{proof}

\subsection{Integrating the volume growth differences}\label{s.integration}  This section considers the difference equation \eref{volode} for the the growth of $|R_t \cap \Box_r(t_0,x_0)|$.  All of the necessary analysis was already carried out by \cite{BuragoIvanovNovikov}, we provide a proof anyway for completeness and to be clear about the choice of the constants $\ep$ and $N$.

By \lref{volumegrowthode} and the relative isoperimetric inequality the difference inequality holds for intervals of length $|I| = \beta r$ 
\begin{equation}\label{e.volode1}
 |\Box_r(x_0) \cap R_t|_d - |\Box_r(x_0) \cap R_{t'}|_d \geq \lambda_1(d)\int_{t'}^t \min\{ |\Box_r \cap R_\tau|_d,| \Box_r \setminus R_\tau|_d\}^{\frac{d-1}{d}} \ d \tau -C\ep r^{d-1}|I|_1 
 \end{equation}
as long as $N(\ep,\beta)$ as in \lref{volumegrowthode} and $r \geq r^*_{N,\ep}(t_0,x_0)$.  Here $\lambda_1(d)$ is the relative isoperimetric constant for cubes in $\R^d$.  We will compare $|\Box_r(x_0) \cap R_t|$ with $r^d \phi(t/r)$ where $\phi$ solves the ODE
 \begin{equation}\label{e.phiode}
  \phi'(t) = \frac{1}{2}\lambda_1 \min\{\phi(t),1-\phi(t)\}^{\frac{d-1}{d}} \ \hbox{ with } \ \phi(0) = 0.
  \end{equation}
 The ODE \eref{phiode} comes, after rescaling, from \eref{volode1} if the error term is ignored and $t-t'$ could be taken arbitrarily small.  The factor of $\frac{1}{2}$ in the front is used to absorb the error terms, any constant smaller than $1$ could be used which would affect the choice of parameters $\ep$ and $N$.  Of course \eref{phiode} does not have uniqueness, so we specify the solution we are interested in
 \[ \phi(t) = \begin{cases} 
 at^d  &\hbox{$0 \leq t \leq b$}\\
 1- a(2b-t)^d &\hbox{$b \leq t \leq 2b$}
 \end{cases}
 \]
 with $a = (\frac{\lambda_1}{2d})^d$ and $b = (\frac{1}{2a})^{1/d} = \frac{d}{\lambda_1}$.  Note that $|\phi'(t)| \leq d a b^{d-1}$, which is just another dimensional constant.  

\begin{lemma}
Suppose that $r \geq r^*_{N,\ep}(t_0,x_0)$ with $N = C(d)M^{5d+2}$ and $\ep = c(d)M^{-(d-1)}$ for appropriately large/small dimensional constants.  Let $t_0 < t_1 < t_2$ be, respectively, the first time that $ |\Box_r \cap R_t|_d \geq \alpha |\Box_r|_d$ and the first time that $ |\Box_r \cap R_t|_d \geq (1-2^{-d}\alpha) |\Box_r|_d$.  Then
\[ |\Box_r \cap R_t|_d \geq r^d \phi((t-t_1)/r) \ \hbox{ for } \ t_1 \leq t \leq t_2,\]
 in particular $t_2$ exists and $t_2 \leq t_1 + \frac{2d}{\lambda_1(d)} r$.
\end{lemma}

Now we can also specify that $\gamma = 1 + \frac{2d}{\lambda_1(d)}$, since $t_1 \leq t_0+\frac{r}{2M} \leq t_0 + 1$, and with the above Lemma $t_2 \leq t_1 + \frac{2d}{\lambda_1(d)}r$.

\begin{proof}
Fix $1>\ep>0$ and $\frac{1}{M} \wedge \frac{1}{4} > \beta > 0$ to be made precise in the course of the proof.  Let $N = \beta^{-2}M^3\ep^{-3}$ and $r \geq r^*_{N,\ep}(t_0,x_0)$ so that \lref{volumegrowthode}, and hence \eref{volode1}, hold on intervals $I$ of length $|I| = \beta r$.

Call $\psi(t) = r^d\phi((t-t_1')/r)$ started at time $t_0 \leq t_1' \leq t_1$ chosen so that $\psi(t_1) = \frac{\alpha}{2}|\Box_r|_d$, precisely,
\[ t_0 \leq t_1' = t_1 - \frac{d\alpha^{1/d}}{2\lambda_1}r\]
 Then $\psi(t_1) <\alpha| \Box_r| \leq | \Box_r \cap R_{t_1}|$.  Let $t_1 < t_* \leq t_2$ be the first time that the inequality $\psi(t) < |\Box_r \cap R_t|_d$ fails (if such time exists).  Note that equality does hold at $t_*$ by \lref{lowercont}.
 
 First we get a lower bound on $t_* - t_1$. Note that, from the lower continuity estimate for $|\Box_r \cap R_t|_d$ \lref{lowercont} and the upper bound on the derivative of $\phi$ and hence $\psi$,
 \[ - M r^{d-1}(t_*-t_1)\leq |\Box_r \cap R_{t_*}|_d - |\Box_r \cap R_{t_1}|_d  \leq \psi(t_*) - \alpha|\Box_r|_d \leq C r^{d-1}(t_*-t_1) - \tfrac{\alpha}{2}|\Box_r|_d. \]
 Rearranging this
 \[ (t_* - t_1) \geq cM^{-1}\alpha r\]
 We will apply \eref{volode1} on an interval of length $\beta r$ with 
 \begin{equation}\label{e.betaconstraint}
  \beta \leq cM^{-1}\alpha.
 \end{equation} 
 to be specified more precisely below (actually matching this upper bound, up to a dimensional constant, will be the right choice).  With this choice the ordering $|\Box_r \cap R_\tau| \geq \psi(\tau)$ holds on the interval $t_* - \beta r \leq \tau \leq t_*$.

 Note that for all $t_* - \beta r \leq \tau \leq t_*$
 \begin{equation}\label{e.complementbd}
  r^d - \psi(\tau) - C\beta r^{d} \leq r^d - \psi(t_*) \leq | \Box_r \setminus R_{t_*}|_d \leq | \Box_r \setminus R_{\tau}|_d +M\beta r^d
 \end{equation}
 again from \lref{lowercont} and the bound on $\psi'$. Also note that on $t_* - \beta r \leq \tau \leq t$
 \begin{equation}\label{e.psi'}
  \psi'(\tau) \geq \frac{1}{2}\lambda_1(d)(2^{-d}\alpha)^{\frac{d-1}{d}}r^{d-1}
  \end{equation}
 since $t_* - \beta r \geq t_1$ by \eref{betaconstraint} so $\psi(\tau) \geq \frac{\alpha}{2}|\Box_r|_d$ and, by definition, $t_* \leq t_2$ so 
 \[ \psi(\tau) \leq \psi(t_*) = |\Box_r \cap R_{t_*}|_d \leq (1-2^{-d}\alpha)|\Box_r|_d.\]
 
 Now finally we have all the necessary set up to compute the ``slope" at the touching time
\begin{align*}
  \psi(t_*) - \psi(t_* - \beta r) &\geq |\Box_r(x_0) \cap R_{t_*}|_d - |\Box_r(x_0) \cap R_{t_*-\beta r}|_d \\
  \hbox{\small{by \eref{volode1}, \eref{betaconstraint}}} \ \ &\geq \lambda_1(d)\int_{t_*-\beta r}^{t_*} \min\{ |\Box_r \cap R_\tau|_d,| \Box_r \setminus R_\tau|_d\}^{\frac{d-1}{d}} \ d \tau -C\ep r^{d-1}|I| \\
  \hbox{\small{by \eref{complementbd}}} \ \ &\geq \lambda_1(d)\int_{t_*-\beta r}^{t_*} \min\{ \psi(\tau) ,r^d - \psi(\tau) - CM\beta r^d\}^{\frac{d-1}{d}} \ d \tau - C \ep r^{d-1}|I| \\
  \hbox{\small{by subadditivity}} \ \ &\geq \int_{t_*-\beta r}^{t_*}\lambda_1(d)\min\{ \psi(\tau) ,r^d - \psi(\tau) \}^{\frac{d-1}{d}}-C [M^{\frac{d-1}{d}}\beta^{\frac{d-1}{d}}+\ep] r^{d-1}\ d \tau \\
  \hbox{\small{by \eref{psi'}}} \ \ &\geq \int_{t_*-\beta r}^{t_*} \psi'(\tau)+ [2^{-d}\lambda_1(d)\alpha^{\frac{d-1}{d}} -C(M^{\frac{d-1}{d}}\beta^{\frac{d-1}{d}}+\ep)] r^{d-1}\ d \tau \\
  & > \int_{t_*-\beta r}^{t_*} \psi'(\tau) d\tau
  \end{align*}
  which is a contradiction as long as we have fixed
  \[ \beta = c(d)M^{-1}\alpha \ \hbox{ and } \ \ep = c(d)\alpha^{\frac{d-1}{d}}\]
  so that the term in square brackets in the second to last line is strictly positive.  Recalling the definition of $\alpha$ from \eref{alphadef} we see, for $M \geq 1$, $\beta =c(d) M^{-(d+1)}$ and $\ep = c(d) M^{-(d-1)}$.  This also finally specifies $N = C(d)M^{5d+2}$. Also note that the constraint \eref{betaconstraint} are satisfied, up to good choices of the dimensional constants $c(d)$.

\end{proof}

\section{Bounds and mixing estimates}\label{s.mixing}

Given a mixing condition on $V$ it is straightforward to apply known result and derive from \tref{main} tail bounds and mixing estimates on $T$.  We will work with the notion of $\alpha$-mixing since that is a standard notion which is more general than finite range dependence.

For each Borel set $U \subset \R_t \times \R^d_x$ define the cylinder $\sigma$-algebras generated by $V$
\[ \mathcal{F}(U) = \sigma(V(t,x): (t,x) \in U).\]
For a pair of $\sigma$-algebras $\mathcal{F}_1$ and $\mathcal{F}_2$ the $\alpha$-mixing coefficient is defined
\[ \alpha(\mathcal{F}_1,\mathcal{F}_2) =  \sup_{A \in \mathcal{F}_1 ,B \in \mathcal{F}_2} |\P(A \cap B) - \P(A)\P(B)|  \] 
Say that $V$ is $\alpha$-mixing in $(t,x)$ if for all diameters $D>0$ the coefficients (abusing notation)
\[ \alpha(r,D) = \sup\{  \alpha( \mathcal{F}(U),\mathcal{F}(U')): \ U, U' \ \hbox{ Borel sets with $d(U,U') \geq r$ with diameter $\leq D$} \}\]
have $\lim_{r \to \infty} \alpha(r,D) = 0$.

We make the assumption of stretched exponential decay of the $\alpha$-mixing coefficients, for some exponent $\beta>0$ and length scale $\ell>0$ and parameters $A,B>0$ we assume
\[ \alpha(r,D) \leq A(1+D)^B\exp(- \ell^{-\beta}r^\beta),\]
and say that $V$ has $\alpha$-mixing with stretched exponential decay with exponent $\beta$ and length scale $\ell$.  We take $\ell = 1$, since the general case can be derived by rescaling.  The constants $c,C$ which appear in the remainder of the section will depend on $A,B$ as well as $d$ and we will not keep track of this dependence any further.

In this case a concentration estimate holds for the spatial averages, from \cite[Proposition 1.9]{DuerinckxGloria}, for any $0 < \ep < 1/2$
\begin{equation}\label{e.concentration}
 \P( |\frac{1}{|Q_r|}\int_{Q_r(t_0,x_0)} V(t,x) \ dt dx| \geq \ep) \leq C \exp(-c \ep^2|\log \ep|^{-\beta'}r^{\beta'}) 
 \end{equation}
with $\beta' = \frac{(d+1)\beta}{d+1+\beta} < \beta$.  Of course the constants in the concentration estimate depend on the constants in the $\alpha$-mixing assumption, but we will not keep track of this dependence.

We now aim to use this estimate to bound the tails of $T(t_0,x_0)$.  By stationarity we can work with $T = T(0,0)$.  Recall
\[ T = r^*_{N,\ep} = \sup \{ r : E_N[V,Q_r] \geq \ep\}\]
with $\ep  = c(d)M^{-(d-1)}$ and $N = \lceil C(d)M^{5d+2} \rceil$ and $E_N$ defined in \eref{en}.

 Applying \eref{concentration} with a union bound gives
 \[ \P(E_N[V;Q_r] \geq \ep) \leq CM^{(5d+2)(d+1)} \exp(-c M^{-2(d-1)}|\log M|^{-\beta'}r^{\beta'}) \]
 We want to control $T$ by a union bound so we need to discretize in $r$. The discretization error is
 \[ |E_N[V;Q_{\lambda r}] - E_N[V;Q_{r}]| \leq CM^{(5d+2)(d+1)}(\lambda-1)^d \leq \ep/2\]
 if we choose $\lambda = 1+cM^{-\frac{(5d+2)(d+1)}{d}-2}$.  So that
 \begin{align*}
  \P(T \geq \tau ) &\leq \sum_{\lambda^k \geq \tau} \P(E_N[V;Q_{\lambda^kr_0}] \geq \ep/2) \\
  &\leq  \sum_{\lambda^k \geq \tau} CM^{(5d+2)(d+1)} \exp(-c M^{-2(d-1)}|\log M|^{-\beta'}\lambda^{k\beta'}) \\
  & \leq C \exp(-c M^{-2(d-1)}|\log M|^{-C}\tau^{\beta'})
  \end{align*}
  where finally we absorbed all the polynomial powers of $M$ in front by changing the power of the logarithm inside the exponential.  This proves the desired tail bounds on $T$.

Next we consider the mixing estimate on $T$. Define the localization, 
\[ r_*^R(t,x) = \sup\{ 0 < r \leq R:  E_N[V,Q_r(t,x)] \geq \ep\}\]
and, for a bounded Borel set $U \subset \R_t \times \R^d_x$,
\[ r_*(U) = \sup_{(t,x) \in U} r_*(t,x).\]
Note that
\[ r^R_*(t,x) \in \mathcal{F}(U+Q_R) \ \hbox{ for all } \ (t,x) \in U.\]
On the event that $\{r_*(U) < R\}$ the localizations $r_*^R(t,x)$ and the actual values of $r_*(t,x)$ agree on $U$.  More precisely, for all $(t,x) \in U$,
\[ r^R_*(t,x) {\bf 1}_{\{r_*(U) < R\}} = r_*(t,x) {\bf 1}_{\{r_*(U) < R\}}. \]
This event has high probability since, by a standard discretization and union bound and the tail bounds established above for $r_*(t,x)$,
\[\P(r_*(U) \geq R ) \leq C(1+\textup{diam}(U))^{d+1} \exp(-c M^{-2(d-1)}|\log M|^{-C}R^{\beta'}).\]
Let $U$ and $U'$ Borel sets with diameter at most $D$ and call $R = \frac{1}{3\sqrt{d+1}}d(U,U')$. Then $U+Q_R$ and $U'+Q_R$ have distance apart at least $\frac{1}{3}d(U,U')$ and, by the mixing condition of $V$
\begin{align*}
 \alpha(\mathcal{F}(U+Q_R),\mathcal{F}(U'+Q_R)) &\leq \alpha(R,D) \\
 &\leq C(1+D+2R)^C \exp(-c R^{\beta}) \\
 &\leq C(1+D)^C \exp(-cR^{\beta'})
 \end{align*}
where the constants $c,C$ may have changed in the last inequality and we used $\beta' < \beta$ to absorb the $C|\log R|$ term in the exponential.

Let random variables $X \in \sigma(r_*|_{U})$ and $Y \in \sigma(r_*|_{U'})$ of the form (abusing notation) $X = X(r_*(t_1,x_1),\dots,r_*(t_n,x_n))$ for some $X: \R^n \to \R$ with $|X| \leq 1$ and points $(t_j,x_j) \in U$, and $Y = Y(r_*(s_1,y_1),\dots,r_*(s_m,y_m))$ for some $Y: \R^m \to \R$ with $|Y| \leq 1$ and points $(s_,y_j) \in U'$.  Call $X^R$ and $Y^R$ to be the same functions applied to the localized $r_*^R$ at the same points.  Then $X$ and $X^R$ agree on $\{r_*(U) < R\}$, and $Y$ and $Y^R$ agree on $\{r_*(U') < R\}$.
\begin{align*}
 |\E[XY] -\E[ X]\E[Y]|   
 &\leq  |\E[ X^RY^R]-\E[ X^R]\E[Y^R]|+|\E[ XY] - \E[ X^RY^R]|\\
 &\quad \quad \quad + | \E[ X^R]\E[Y^R]-\E[X]\E[Y]| \\
 &\leq  \alpha(R,D)+4\P(r_*(U)>R) + 4\P(r_*(U')>R)\\
 &\leq  C(1+D)^C \exp(-c M^{-2(d-1)}|\log M|^{-C}d(U,U')^{\beta'})
 \end{align*}
 This establishes the $\alpha$-mixing rate for the field $T = r_*$.

  \bibliographystyle{plain}
\bibliography{geqn_art}
\end{document}